\documentclass[graybox]{svmult}
\usepackage{amssymb,amsmath,enumerate,amscd,graphicx,color}
\usepackage[usenames,dvipsnames,svgnames,table]{xcolor}
\usepackage{pdfsync}
 \usepackage[colorlinks=true, pdfstartview=FitV, linkcolor=blue, citecolor=blue, urlcolor=blue]{hyperref}

\newcommand{\End}{\text{\rm End}}
\newcommand{\thmref}[1]{Theorem~\ref{#1}}

\newcommand{\lemref}[1]{Lemma~\ref{#1}}
\newcommand{\eqnref}[1]{~(\ref{#1})}

\newcommand{\germ}{\mathfrak}
\newtheorem{thm}{Theorem}[section]
\newtheorem{lem}[thm]{Lemma}

\begin{document}
\title*{Free field realizations of the Date-Jimbo-Kashiwara-Miwa algebra}
\author{Ben Cox, Vyacheslav Futorny and Renato Alessandro Martins }
\institute{Ben Cox \at Department of Mathematics,
The College of Charleston,
66 George Street,
Charleston SC 29424, USA \email{coxbl@cofc.edu} \and 
Vyacheslav Futorny \at Departamento de Matem\'atica,
Instituto de Matem\'atica e Estat\'istica,
S\~ao Paulo, Brasil
\email{vfutorny@gmail.com} \and Renato Alessandro Martins \at Departamento de Matem\'atica,
Instituto de Matem\'atica e Estat\'istica,
S\~ao Paulo, Brasil 
\email{renatoam@ime.usp.br}}

\maketitle
\abstract*{We use the description of the universal central extension of the DJKM algebra $\mathfrak{sl}(2,  R)$ where $ R=\mathbb C[t,t^{-1},u\,|\,u^2=t^4-2ct^2+1 ]$ given in \cite{MR2813377} to construct realizations of the DJKM algebra in terms of sums of partial differential operators.}
\abstract{We use the description of the universal central extension of the DJKM algebra $\mathfrak{sl}(2,  R)$ where $ R=\mathbb C[t,t^{-1},u\,|\,u^2=t^4-2ct^2+1 ]$ given in \cite{MR2813377} to construct realizations of the DJKM algebra in terms of sums of partial differential operators.}


%
\section{Introduction}  
 
Suppose $R$ is a commutative algebra and $\mathfrak g $ is a simple Lie algebra, both defined over the complex numbers.  From the work of C. Kassel and J.L. Loday (see \cite{MR694130}, and \cite{MR772062}) it is shown that the universal central extension $\hat{{\mathfrak g}}$ of ${\mathfrak g}\otimes R$ is the vector space $\left({\mathfrak g}\otimes R\right)\oplus \Omega_R^1/dR$ where $\Omega_R^1/dR$ is the space of K\"ahler differentials modulo exact forms (see \cite{MR772062}).  Let  $\overline{a}$ denote the image of $a\in\Omega^1_R$ in the quotient $\Omega^1_R/dR$ and let $(-,-)$ denote the Killing form on $\mathfrak g$.   Then more precisely the universal central extension $\hat{{\mathfrak g}}$ is the vector space  $\left({\mathfrak g}\otimes R\right)\oplus \Omega_R^1/dR$ made into a Lie algebra by defining
$$
[x\otimes f,y\otimes g]:=[xy]\otimes fg+(x,y)\overline{fdg},\quad [x\otimes f,\omega]=0
$$
for $x,y\in\mathfrak g$, $f,g\in R$,  and $\omega\in \Omega_R^1/dR$.        A natural and useful question comes to mind
as to whether there exists free field or Wakimoto type realizations of these algebras.  M. Wakimoto and B. Feigin and E. Frenkel answered this quesiton when $R$ is the ring of Laurent polynomials in one variable (see \cite{W} and \cite{MR92f:17026}).  The goal of this paper is to describe such a realization for the universal central extension of $\mathfrak g=\mathfrak{sl}(2,R)$ where $R=\mathbb C[t,t^{-1},u|u^2=t^4-2ct^2+1]$, $c\neq \pm 1$, is a DJKM algebra.

   In Kazhdan and Luszig's explicit study of the tensor structure of modules for affine Lie algebras (see \cite{MR1186962} and \cite{MR1104840}) the ring of functions regular everywhere except at a finite number of points appears naturally.   This algebra M. Bremner gave the name {\it $n$-point algebra}.  In particular in the monograph  \cite[Ch. 12]{MR1849359} algebras of the form $\oplus _{i=1}^n\mathfrak g((t-x_i))\oplus\mathbb Cc$ appear in the description of the conformal blocks.  These contain the $n$-point algebras $ \mathfrak g\otimes \mathbb C[(t-x_1)^{-1},\dots, (t-x_N)^{-1}]\oplus\mathbb Cc$ modulo part of the center $\Omega_R/dR$.   M. Bremner explicitly described the universal central extension of such an algebra in \cite{MR1261553}.  
   
   Let $R$ denote the ring of rational functions on the Riemann sphere   $S^2=\mathbb  C\cup\{\infty\}$ with poles only in the set  of distinct points $\{a_1,a_2,a_3,\infty \}\subset S^2$.  In the literature one can find the fact that the automorphism group $PGL_2(\mathbb C)$ of $\mathbb C(s)$ is simply 3-transitive and $R$ is a subring of $\mathbb C(s)$, so that $R$ is isomorphic to the ring of rational functions with poles at $\{\infty,0,1,a\}$.  This isomorphism motivates one setting $a=a_4$ and then defining the {\it $4$-point ring} as $R=R_a=\mathbb C[s,s^{-1},(s-1)^{-1},(s-a)^{-1}]$ where $a\in\mathbb C\backslash\{0,1\}$.    Letting $S:=S_b=\mathbb C[t,t^{-1},u]$ where $u^2=t^2-2bt+1$ with $b$ a complex number not equal to $\pm 1$,   M. Bremner has shown us that $R_a\cong S_b$. 
The latter ring is $\mathbb Z_2$-graded where $t$ is even and $u$ is odd,  and is a cousin to super Lie algebras, so this ring lends itself to the techniques of conformal field theory.  M. Bremner gave an explicit  description of the universal central extension of $\mathfrak g\otimes R$, in terms of ultraspherical (Gegenbauer) polynomials where $R$ is this four point algebra (see \cite{MR1249871}).      Motivated by talks with M. Bremner, the first author gave free field realizations for the four point algebra where the center acts nontrivially (see \cite{MR2373448}). 

In Bremner's study of the elliptic affine Lie algebras, $\mathfrak{sl}(2, R)  \oplus\left( \Omega_R/dR\right)$ where $R=\mathbb C[x,x^{-1},y\,|\,y^2=4x^3-g_2x-g_3]$,  he next explicitly described the universal central extension of this algebra in terms of Pollaczek polynomials (see  \cite{MR1303073}).   Variations of these algebras appear in recent work of A. Fialowski and M. Schlichenmaier \cite{FailS} and \cite{MR2183958}.  Together with Andr\'e Bueno, the first and second authors of the present article described free field type realizations of the elliptic Lie algebra where $R=\mathbb C[t,t^{-1},u\,|, u^2=t^3-2bt^2-t]$, $b\neq \pm 1$ (see \cite{MR2541818}).

In  \cite{MR701334}  Date, Jimbo, Kashiwara and Miwa described
integrable systems arising from the Landau-Lifshitz differential equation. The integrable hierarchy of this equation was shown to be written in terms of free fermions defined on an elliptic curve. These authors introduced an 
infinite-dimensional Lie algebra which is a one dimensional central extension of $\mathfrak g\otimes \mathbb C[t,t^{-1},u|u^2=(t^2-b^2)(t^2-c^2)]$ where $b\neq \pm c$ are complex
 constants and $\mathfrak g$ is a simple finite dimensional Lie algebra. This algebra, which we call the DJKM algebra,  acts on the space of solutions of the Landau-Lifshitz equation as infinitesimal B\"acklund transformations.

In \cite{MR2813377} the first and second authors provided the commutations relations of
the universal central extension of the DJKM Lie algebra in terms of a basis of the algebra and certain polynomials.  More precisely in order to pin down this central extension, one needed to describe four \color{black} families of polynomials that appeared as coefficients in the commutator formulae.  Two of these families of polynomials are given in terms of
elliptic integrals and the other two families are slight variations of ultraspherical polynomials.  
One of the families is precisely given by the associated ultraspherical polynomials (see \cite{MR663314}).   The associated ultraspherical polynomials in turn are, up to factors of ascending factorials, particular associated Jacobi polynomials.   The associated Jacobi polynomials are known to satisfy certain fourth order linear differential equations (see \cite{MR2191786} and formula (48) in \cite{MR915027}).    In \cite{MR3090080} the first and second authors together with Juan Tirao, showed that the remaining family of polynomials are not of classical type and are orthogonal.   The aim of the present article is to provide free field type realizations of the DJKM algebra. 

The Lie algebra above is an example of a Krichever-Novikov  algebra (see (\cite{MR902293}, \cite{MR925072}, \cite{MR998426}).  A fair amount of interesting and
fundamental work has be done by Krichever, Novikov, Schlichenmaier, and Sheinman on the representation theory of these algebras.
 In particular Wess-Zumino-Witten-Novikov theory and analogues of the Knizhnik-Zamolodchikov equations are developed for  these algebras
(see the survey article \cite{MR2152962}, and for example \cite{MR1706819}, \cite{MR1706819}, \cite{MR2072650}, \cite{MR2058804}, \cite{MR1989644}, and \cite{MR1666274}).

If these realizations had been constructed in a logically historical fashion one would have first described them for a three point algebra.  This is  the case where $R$ denotes the ring of rational functions with poles only in the set of distinct complex numbers $\{a_1,a_2,a_3\}$. This algebra is isomorphic to $\mathbb C[s,s^{-1},(s-1)^{-1}]$. M. Schlichenmaier gave an isomorphic description of the three point algebra as $\mathbb C[(z^2-a^2)^k,z(z^2-a^2)^k\,|\, k\in\mathbb Z]$ where $a\neq 0$ (see \cite{MR2058804}). 
E. Jurisich and the first author of the present paper proved that $R\cong \mathbb C[t,t^{-1},u\,|\,u^2=t^2+4t]$ so that the three point algebra looks more like $S_b$ above.
The main result of \cite{CoxJur}  provides a natural free field realization in terms of a $\beta$-$\gamma$-system and the oscillator algebra of the three point affine Lie algebra when $\mathfrak g=\mathfrak{sl}(2,\mathbb C)$.   Besides M. Bermner's article mentioned above, other work on the universal central extension of $3$-point algebras can be found in \cite{MR2286073}. 
The three point algebra is probably  the simplest non-trivial example of a Krichever-Novikov algebra beyond an affine Kac-Moody algebra 
 (see \cite{MR902293}, \cite{MR925072}, \cite{MR998426}).  Interesting foundational and 
fundamental research has be done by Krichever, Novikov, Schlichenmaier, and Sheinman on the representation theory of the Krichever-Novikov algebras. 
 In particular Wess-Zumino-Witten-Novikov theory and analogues of the Knizhnik-Zamolodchikov equations are developed for  these algebras 
(see the survey article \cite{MR2152962}, and for example \cite{MR1706819}, \cite{MR2072650}, \cite{MR2058804}, \cite{MR1989644}, and \cite{MR1666274}).

Previous related work on highest weight modules of $\mathfrak{sl}(2,R)$ written in terms of infinite sums of partial differential operators is be found in the early paper of H. P. Jakobsen and V. Kac \cite{JK}.

M. Wakimoto's motivation for the use of a free field realization was to prove a conjecture of V. Kac and D. Kazhdan on the character of certain irreducible representations of affine Kac-Moody algebras at the critical level (see \cite{W} and \cite{MR2146349}). Interestingly free field realizations have been used by V.~V. Schechtman and A.~N. Varchenko (and others) to provide integral solutions to the KZ-equations (see for example \cite{MR1077959} and \cite{MR1629472} and their references).  Another application is to help in describing the center of a certain completion of the enveloping algebra of an affine Lie algebra at the critical level and this determination of the center is an important ingredient in the formulation of the geometric Langland's correspondence \cite{MR2332156}.  As a last bit of motivation, the work  B. Feigin and E. Frenkel has shown that free field realizations of an affine Lie algebra appear naturally in the context of the generalized AKNS hierarchies \cite{MR1729358}.

\section{Description of the universal central extension of DJKM algebras}
As we described in the introduction, C. Kassel showed the universal central extension of the current algebra $\mathfrak g\otimes R$ where $\mathfrak g$ is a simple finite dimensional Lie algebra defined over $\mathbb C$, is the vector space $\hat{\mathfrak g}=(\mathfrak g\otimes R)\oplus \Omega_R^1/dR$ with Lie bracket given by
$$
[x\otimes f,Y\otimes g]=[xy]\otimes fg+(x,y)\overline{fdg},  [x\otimes f,\omega]=0,  [\omega,\omega']=0,
$$
  where $x,y\in\mathfrak g$, and $\omega,\omega'\in \Omega_R^1/dR$ and $(x,y)$  denotes the Killing  form  on $\mathfrak g$.

Consider the polynomial
$$
p(t)=t^n+a_{n-1}t^{n-1}+\cdots+a_0
$$
where $a_i\in\mathbb C$ and $a_n=1$. 
Fundamental to the explicit description of the universal central extension for $R=\mathbb C[t,t^{-1},u|u^2=p(t)]$ is the following:
\begin{theorem}[\cite{MR1303073},Theorem 3.4]  Let $R$ be as above.  The set 
$$
\{\overline{t^{-1}\,dt},\overline{t^{-1}u\,dt},\dots, \overline{t^{-n}u\,dt}\}
$$
 forms a basis of $\Omega_R^1/dR$ (omitting $\overline{t^{-n}u\ dt}$ if $a_0=0$).   
\end{theorem}
%

\begin{lem}[\cite{MR2813377}]  If $u^m=p(t)$ and $R=\mathbb C[t,t^{-1},u|u^m=p(t)]$, then in $\Omega_R^1/dR$, one has
\begin{equation}\label{recursionreln}
((m+1)n+im)t^{n+i-1}u\,dt \equiv - \sum_{j=0}^{n-1}((m+1)j+mi)a_jt^{i+j-1}u\,dt\mod dR
\end{equation}
\end{lem}
%
%
%

In the DJKM algebra setting one takes $m=2$ and $p(t)=(t^2-a^2)(t^2-b^2)$ with $a\neq \pm b$ and neither $a$ nor $b$ is zero.  The Lemma above leads one to introduce the polynomials  $P_k:=P_k(c)$ in $c=(a^2+b^2)/2$ satisfy the recursion relation 
$$
(6+2k)P_k(c)   
=4k cP_{k-2}(c)-2(k-3)P_{k-4}(c)
$$
for $k\geq 0$.
Setting 
$$
P(c,z):=\sum_{k\geq -4}P_k(c)z^{k+4}=\sum_{k\geq 0}P_{k-4}(c)z^{k}.
$$
one proves in \cite{MR2813377} that
\begin{align}\label{funde}
\frac{d}{dz}P(c,z)-&\frac{3z^4-4c z^2+1}{z^5-2cz^3+z}P(c,z) \\
&=\frac{2\left(P_{-1}+cP_{-3} \right)z^3 +P_{-2} z^2+(4cz^2-1)P_{-4} }{z^5-2cz^3+z} \notag
\end{align}
%
\subsection{Elliptic Case 1}
Taking the initial conditions $P_{-3}(c)=P_{-2}(c)=P_{-1}(c)=0$ and $P_{-4}(c)=1$ then we arrive at a generating function 
$$
P_{-4}(c,z):=\sum_{k\geq -4}P_{-4,k}(c)z^{k+4}=\sum_{k\geq 0}P_{-4,k-4}(c)z^{k},
$$
expressed in terms of an elliptic integral
\begin{align*}
P_{-4}(c,z)&=z\sqrt{1-2 c z^2+z^4}\int \frac{4cz^2-1}{z^2(z^4-2c z^2+1)^{3/2}}\, dz.
\end{align*}
The way that we interpret the right hand integral is to expand $(z^4-2c z^2+1)^{-3/2}$ as a Taylor series about $z=0$ and then formally integrate term by term and multiply the result by the Taylor series of $z\sqrt{1-2 c z^2+z^4}$.    More precisely one integrates formally with zero constant term
 $$
 \int (4c-z^{-2})\sum_{n=0}^\infty Q_n^{(3/2)}(c)z^{2n}\,dz =\sum_{n=0}^\infty \frac{4cQ_n^{(3/2)}(c)}{2n+1}z^{2n+1} -\sum_{n=0}^\infty \frac{Q_n^{(3/2)}(c)}{2n-1}z^{2n-1}
 $$ 
 where $Q_n^{(\lambda)}(c)$ is the $n$-th Gegenbauer polynomial.
When multiplying this by 
$$
z\sqrt{1-2cz^2+z^4}=\sum_{n=0}^\infty Q_n^{(-1/2)}(c)z^{2n+1}
$$
one obtains the series $P_{-4}(c,z)$.

\subsection{Elliptic Case 2}
If we take initial conditions $P_{-4}(c)=P_{-3}(c)=P_{-1}(c)=0$ and $P_{-2}(c)=1$ then we arrive at a generating function defined in terms of another elliptic integral:
\begin{align*}
P_{-2}(c,z)&=z\sqrt{1-2 c z^2+z^4}\int \frac{1}{ (z^4-2c z^2+1)^{3/2}}\, dz.
\end{align*}

\subsection{Gegenbauer  Case 3}
If we take $P_{-1}(c)=1$, and $P_{-2}(c)=P_{-3}(c)=P_{-4}(c)=0$ and set 
$$
P_{-1}(c,z)=\sum_{n\geq 0}P_{-1,n-4}z^n,
$$
then we get a solution which after solving for the integration constant can be turned into a power series solution 
\begin{align*}
P_{-1}(c,z)
&=\frac{1}{c^2-1}\left(cz-z^3-cz+c^2z^3-\sum_{k=2}^\infty c Q_n^{(-1/2)}(c)z^{2n+1}\right)  
\end{align*}
where $Q^{(-1/2)}_n(c)$ is the $n$-th Gegenbauer polynomial.   Hence
\begin{align*}
P_{-1,-4}(c)&=P_{-1,-3}(c)=P_{-1,-2}(c) =P_{-1,2m}(c)=0,\quad P_{-1,-1}(c)=1,  \\
P_{-1,2n-3}(c)&=\frac{-cQ_{n}(c)}{c^2-1},
\end{align*}
for $m\geq 0$ and $n\geq 2$ .

\subsection{Gegenbauer Case 4}
Next we consider the initial conditions $P_{-1}(c)=0=P_{-2}(c)=P_{-4}(c)=0$ with $P_{-3}(c)=1$ and set 
$$
P_{-3}(c,z)=\sum_{n\geq 0}P_{-3,n-4}(c)z^n.
$$
then we get a power series solution
\begin{align*}
P_{-3}(c,z)
&=\frac{1}{c^2-1}\left(c^2z-cz^3-z+cz^3-\sum_{k=2}^\infty Q_n^{(-1/2)}(c)z^{2n+1}\right)  \\
\end{align*}
where $Q^{(-1/2)}_n(c)$ is the $n$-th Gegenbauer polynomial.  Hence
\begin{align*}
P_{-3,-4}(c)&=P_{-3,-2}(c)=P_{-3,-1}(c) =P_{-1,2m}(c)=0,\quad P_{-3,-3}(c)=1,  \\
P_{-3,2n-3}(c)&=\frac{-Q_{n}(c)}{c^2-1},
\end{align*}
for $m\geq 0$ and $n\geq 2$.
\section{Generators and relations}

Set $\omega_0=\overline{t^{-1}\,dt}$,  and $\omega_{-k}=\overline{t^{-k}u\,dt}$, $k=1,2,3,4$.

\begin{theorem}[see  \cite{MR2813377}] Let $\mathfrak g$ be a simple finite dimensional Lie algebra over the complex numbers with   the Killing form $(\,|\,)$ and define $\psi_{ij}(c)\in\Omega_R^1/dR$ by
\begin{equation}
\psi_{ij}(c)=\begin{cases} 
\omega_{i+j-2}&\quad \text{ for }\quad i+j=1,0,-1,-2 \\
P_{-3,i+j-2}(c) (\omega_{-3}+c\omega_{-1})&\quad \text{for} \quad i+j =2n-1\geq 3,\enspace n\in\mathbb Z, \\
P_{-3,i+j-2}(c) (c\omega_{-3}+\omega_{-1})&\quad \text{for} \quad i+j =-2n+1\leq - 3, n\in\mathbb Z, \\
P_{-4,|i+j|-2}(c) \omega_{-4} +P_{-2,|i+j|-2}(c)\omega_{-2}&\quad\text{for}\quad |i+j| =2n \geq 2, n\in\mathbb Z. \\
\end{cases}
\end{equation}
The universal central extension of the Date-Jimbo-Kashiwara-Miwa  algebra is the $\mathbb Z_2$-graded Lie algebra 
$$
\widehat{\mathfrak g}=\widehat{\mathfrak g}^0\oplus \widehat{\mathfrak g}^1,
$$
where
\begin{align*}
\widehat{\mathfrak g}^0&=\left(\mathfrak g\otimes \mathbb C[t,t^{-1}]\right)\oplus \mathbb C\omega_{0}, \\ \widehat{\mathfrak g}^1&=\left(\mathfrak g\otimes \mathbb C[t,t^{-1}]u\right)\oplus \mathbb C\omega_{-4}\oplus \mathbb C\omega_{-3}\oplus \mathbb C\omega_{-2}\oplus \mathbb C\omega_{-1}
\end{align*}
with bracket
\begin{align*}
[x\otimes t^i,y\otimes t^j]&=[x,y]\otimes t^{i+j}+\delta_{i+j,0}j(x,y)\omega_0, \\ \\
[x\otimes t^{i-1}u,y\otimes t^{j-1}u]&=[x,y]\otimes (t^{i+j+2}-2ct^{i+j}+t^{i+j-2}) \\
 &\quad +\left(\delta_{i+j,-2}(j+1) -2cj\delta_{i+j,0} +(j-1)\delta_{i+j,2}\right)(x,y)\omega_0, \\ \\
[x\otimes t^{i-1}u,y\otimes t^{j}]&=[x,y]u\otimes t^{i+j-1}+ j(x,y)\psi_{ij}(c).
\end{align*}

\end{theorem}
The theorem above is similar to the results that M.  Bremner obtained for the elliptic and four point affine Lie algebra cases (\cite[Theorem 4.6]{MR1303073} and \cite[Theorem 3.6]{MR1249871} respectively) and with the isomorphism obtained for the three point algebra given in \cite{CoxJur}.

\begin{theorem}\label{3ptthm}
The universal central extension of the algebra $\mathfrak{sl}(2,\mathbb C)\otimes \mathcal R$ is isomorphic to the Lie algebra with generators $e_n$, $e_n^1$, $f_n$, $f_n^1$, $h_n$, $h_n^1$, $n\in\mathbb Z$, $\omega_0$, $\omega_{-1}$, $\omega_{-2}$, $\omega_{-3}$, $\omega_{-4}$ and relations given by
\begin{align}
[x_m,x_n]&:=[x_m,x_n^1]=[x_m^1,x_n^1]=0,\quad \text{ for }x=e,f\label{xs}\\
[h_m,h_n]&:=-2m\delta_{m,-n}\omega_0=(n-m)\delta_{m,-n}\omega_0 , \\
[h^1_m,h^1_n]&:=2\left((n+2)\delta_{m+n,-4}-2c(n+1)\delta_{m+n,-2}+n\delta_{m+n,0}\right)\omega_0,\\
[h_m,h_n^1]&:=-2m\psi_{mn}(c),\\
 [\omega_i,x_m]&=[\omega_i,\omega_j]=0,\quad \text{ for }x=e,f,h,\quad i,j\in\{0,1\} \label{omega1} \\ 
[e_m,f_n]&=h_{m+n}-m\delta_{m,-n}\omega_0, \label{ef}\\
[e_m,f_n^1]&=h^1_{m+n}-m\psi_{mn}(c)=:[e_m^1,f_n], \\
[e_m^1,f_n^1]&:=h_{m+n+4}-2ch_{m+n+2}+h_{m+n}\\
&\quad+\left((n+2)\delta_{m+n,-4}-2c(n+1)\delta_{m+n,-2}+n\delta_{m+n,0}\right)\omega_0,\notag  \\
[h_m,e_n]&:=2e_{m+n}, \\
[h_m,e^1_n]&:=2e^1_{m+n} = :[h_m^1,e_m],  \\
[h_m^1,e_n^1]&:=2e_{m+n+4}-4ce_{m+n+2}+2e_{m+n},\label{he}\\
[h_m,f_n]&:=-2f_{m+n}, \\
[h_m,f^1_n]&:=-2f^1_{m+n} =:[h_m^1,f_m], \\
[h_m^1,f_n^1]&:=-2f_{m+n+4}+4cf_{m+n+2}-2f_{m+n} ,\label{last}
\end{align}
for all $m,n\in\mathbb Z$.
\end{theorem}

\begin{proof}  Let $\mathfrak f$ denote the free Lie algebra with the generators  $e_n$, $e_n^1$, $f_n$, $f_n^1$, $h_n$, $h_n^1$, $n\in\mathbb Z$, $\omega_0$, $\omega_{-1}$, $\omega_{-2}$, $\omega_{-3}$, $\omega_{-4}$ and relations given above \eqnref{xs}-\eqnref{last}.  The map
$\phi:\mathfrak f\to (\mathfrak{sl}(2,\mathbb C)\otimes \mathcal R)\oplus ( \Omega_{\mathcal R}/d\mathcal R)$ given by 
\begin{align*}
\phi(e_n):&=e\otimes t^n,\quad \phi(e_n^1)=e\otimes ut^n, \\
\phi(f_n):&=f\otimes t^n,\quad \phi(f_n^1)=f\otimes ut^n, \\
\phi(h_n):&=h\otimes t^n,\quad \phi(h_n^1)=h\otimes ut^n, \\
\phi(\omega_0):&=\overline{t^{-1}\,dt},\quad \phi(\omega_{-1})=\overline{t^{-1}u\,dt}, \\
\phi(\omega_{-2}):&=\overline{t^{-2}u\,dt},\quad \phi(\omega_{-3})=\overline{t^{-3}u\,dt}, \\
\phi(\omega_{-4}):&=\overline{t^{-4}u\,dt},\\
\end{align*}
for $n\in\mathbb Z$ is a surjective Lie algebra homomorphism.  

 Consider the subalgebras 
 \begin{align*}
 S_+&=\langle e_n,e_n^1\,|\,n\in\mathbb Z\rangle  \\
 S_0&=\langle h_n,h_n^1,\omega_0,\omega_{-1},\omega_{-2},\omega_{-3},\omega_{-4}\,|\,n\in\mathbb Z\rangle  \\
S_-&=\langle f_n,f_n^1\,|\,n\in\mathbb Z\rangle
\end{align*}
 where $\langle X\enspace\rangle$ means spanned by the set $X$ and set $S=S_-+S_0+S_+$.   By \eqnref{xs} -\eqnref{omega1} we have
$$
S_+=\sum_{n\in\mathbb Z}\mathbb Ce_n+\sum_{n\in\mathbb Z}\mathbb Ce_n^1,\quad S_-=\sum_{n\in\mathbb Z}\mathbb C f_n+\sum_{n\in\mathbb Z}\mathbb Cf_n^1
$$
$$
S_0=\sum_{n\in\mathbb Z}\mathbb C h_n+\sum_{n\in\mathbb Z}\mathbb Ch_n^1+\mathbb C\omega_0+\mathbb C\omega_{-1}+\mathbb C\omega_{-2}+\mathbb C\omega_{-3}+\mathbb C\omega_{-4}
$$
By \eqnref{ef}-\eqnref{he} we see that 
\begin{gather*}
[e_n,S_+]=[e_n^1,S_+]=0,\quad [h_n,S_+]\subseteq S_+,\quad [h_n^1,S_+]\subseteq S_+,\\ [f_n,S_+]\subseteq S_0,\quad [f_n^1,S_+]\subseteq S_0.
\end{gather*}
and similarly 
$[x_n,S_-]=[x_n^1,S_-] \subseteq  S$, $[x_n,S_0]=[x_n^1,S_0] \subseteq  S$ for $x=e,f,h$. 
To sum it up we observe that $[x_n,S]\subseteq S$ and $[x_n^1,S]\subseteq S$ for $n\in\mathbb Z$, $x=h,e,f$.
Thus $[S,S]\subset S$.
Hence $S$ contains the generators of $\mathfrak f$ and is a subalgebra.  Hence $S=\mathfrak f$.  One can now see that $\phi$ is a Lie algebra isomorphism.
\end{proof}

\section{A triangular decomposition of DJKM loop algebras $\mathfrak g\otimes R$}

From now on we identify $R_a$ with $\mathcal S$ and set $R=\mathcal S$ which has a basis $t^i,t^iu$, $i\in\mathbb Z$.  Let $p:R\to R$ be the automorphism given by $p(t)=t$ and $p(u)=-u$.   Then one can decompose $R=R^0\oplus R^1$ where $R^0=\mathbb C[t^{\pm1}]=\{r\in R\,|\, p(r)=r\}$ and
$R^1=\mathbb C[t^{\pm1}]u=\{r\in R\,|\, p(r)=-r\}$ are the eigenspaces of $p$.
From now on $\mathfrak g$ will denote a simple Lie algebra over $\mathbb C$ with triangular decomposition $\mathfrak g=\mathfrak n_-\oplus \mathfrak h\oplus\mathfrak n_+$ and then the {\it DJKM loop algebra} $L({\mathfrak g}):=\mathfrak g\otimes R$ has a corresponding $\mathbb Z/2\mathbb Z$-grading:  $L({\mathfrak g})^i:=\mathfrak g\otimes R^i$ for $i=0,1$.  However the degree of $t$ does not render $L({\mathfrak g})$ a $\mathbb Z$-graded Lie algebra.  This leads one to the following notion.

Suppose $I$ is an additive subgroup of the rational numbers $\mathbb P$ and $\mathcal A$ is a $\mathbb C$-algebra such that $\mathcal A=\oplus_{i\in I}\mathcal A_i$  and there exists a fixed $l\in\mathbb N$, with
$$
\mathcal A_i\mathcal A_j\subset \oplus_{|k-(i+j)|\leq l}\mathcal A_k
$$
for all $i,j\in\mathbb Z$.  Then $\mathcal A$ is said to be an {\it $l$-quasi-graded algebra}.  For $0\neq x\in \mathcal A_i$ one says that $x$ is {\it homogeneous of degree $i$} and one writes $\deg x=i$.

  For example $R$ has the structure of a $1$-quasi-graded algebra where $I=\frac{1}{2}\mathbb Z$ and $\deg t^i=i$, $\deg t^iu=i+\frac{1}{2}$.

A {\it weak triangular decomposition} of a Lie algebra $\mathfrak l$ is a triple $(\mathfrak H,\mathfrak l_+,\sigma)$ satisfying
\begin{enumerate}
\item $\mathfrak H$ and $\mathfrak l_+$ are subalgebras of $\mathfrak l$,
\item $\mathfrak H$ is abelian and $[\mathfrak H,\mathfrak l_+]\subset \mathfrak l_+$,
\item $\sigma$ is an anti-automorphism of $\mathfrak l$ of order $2$ which is the identity on $\mathfrak h$ and 
\item $\mathfrak l=\mathfrak l_+\oplus \mathfrak H\oplus\sigma( \mathfrak l_+)$.
\end{enumerate}
We will let $\sigma(\mathfrak l_+)$ be denoted by $\mathfrak l_-$.

\begin{theorem} The DJKM loop algebra $L({\mathfrak g})$ is $1$-quasi-graded Lie algebra where $\deg (x\otimes f)=\deg f$ for $f$ homogeneous.  Set $R_+=\mathbb C(1+u)\oplus \mathbb C[t,u]t$ and $R_-=p(R_+)$.  Then $L({\mathfrak g})$ has a weak triangular decomposition given by
$$
L({\mathfrak g})_\pm=\mathfrak g\otimes R_\pm,\quad \mathcal H:=\mathfrak h\otimes \mathbb C.
$$
\end{theorem}
\begin{proof}  This is essentially the same proof as \cite{MR1249871}, Theorem 2.1 and so will be omitted.
\end{proof}

\subsection{Formal Distributions}
We need some more notation that will make some of the arguments
later more transparent.
Our notation follows roughly \cite{MR99f:17033} and \cite
{MR2000k:17036}:  The {\it formal (Dirac) delta function}
$\delta(z/w)$ is the formal distribution
$$
\delta(z/w)=z^{-1}\sum_{n\in\mathbb Z}z^{-n}w^{n}=w^{-1}\sum_{n\in\mathbb Z}z^{n}w^{-n}.
$$
For any sequence of elements $\{a_{m}\}_{m\in
\mathbb Z}$ in the ring $\End (V)$, $V$ a vector space,  the
{\it formal distribution}
\begin{align*}
a(z):&=\sum_{m\in\mathbb Z}a_{m}z^{-m-1}
\end{align*}
is called a {\it field}, if for any $v\in V$, $a_{m}v=0$ for $m\gg0$.
If $a(z)$ is a field, then we set
\begin{align}\label{usualnormalordering}
    a(z)_-:&=\sum_{m\geq 0}a_{m}z^{-m-1},\quad\text{and}\quad
   a(z)_+:=\sum_{m<0}a_{m}z^{-m-1}.
\end{align}
 The {\it normal ordered product} of two distributions
$a(z)$ and
$b(w)$ (and their coefficients) is
defined by
\begin{equation}\label{normalorder}
\sum_{m\in\mathbb Z}\sum_{n\in\mathbb
Z}:a_mb_n:z^{-m-1}w^{-n-1}=:a(z)b(w):=a(z)_+b(w)+b(w)a(z)_-.
\end{equation}

%
Then one defines recursively
\[
:a^1(z_1)\cdots a^k(z_k):=:a^1(z_1)\left(:a^2(z_2)\left(:\cdots
:a^{k-1}(z_{k-1}) a^k(z_k):\right)\cdots
:\right):,
\]
while normal ordered product
\[
:a^1(z)\cdots
a^k(z):=\lim_{z_1,z_2,\cdots, z_k\to
z} :a^1(z_1)\left(:a^2(z_2)\left(:\cdots :a^{k-1}(z_{k-1})
a^k(z_k):\right)\cdots
\right):
\]
will only be defined for certain $k$-tuples $(a^1,\dots,a^k)$.

Set 
\begin{equation}\label{contraction}
\lfloor
ab\rfloor=a(z)b(w)-:a(z)b(w):= [a(z)_-,b(w)],
\end{equation}
which is called the {\it contraction} of the two formal distributions 
$a(z)$ and $b(w)$. 
\begin{theorem}[Wick's Theorem, \cite{MR85g:81096}, \cite{MR99m:81001} or  
\cite{MR99f:17033} ]  Let  $a^i(z)$ and $b^j(z)$ be formal
distributions with coefficients in the associative algebra 
 $\End(\mathbb C[\mathbf x]\otimes \mathbb C[\mathbf y])$, 
 satisfying
\begin{enumerate}
\item $[ \lfloor a^i(z)b^j(w)\rfloor ,c^k(x)_\pm]=[ \lfloor
a^ib^j\rfloor ,c^k(x)_\pm]=0$, for all $i,j,k$ and
$c^k(x)=a^k(z)$ or
$c^k(x)=b^k(w)$.
\item $[a^i(z)_\pm,b^j(w)_\pm]=0$ for all $i$ and $j$.
\item The products 
$$
\lfloor a^{i_1}b^{j_1}\rfloor\cdots
\lfloor a^{i_s}b^{i_s}\rfloor:a^1(z)\cdots a^M(z)b^1(w)\cdots
b^N(w):_{(i_1,\dots,i_s;j_1,\dots,j_s)}
$$
have coefficients in
$\End(\mathbb C[\mathbf x]\otimes \mathbb C[\mathbf y])$ for all subsets
$\{i_1,\dots,i_s\}\subset \{1,\dots, M\}$, $\{j_1,\dots,j_s\}\subset
\{1,\cdots N\}$. Here the subscript
${(i_1,\dots,i_s;j_1,\dots,j_s)}$ means that those factors $a^i(z)$,
$b^j(w)$ with indices
$i\in\{i_1,\dots,i_s\}$, $j\in\{j_1,\dots,j_s\}$ are to be omitted from
the product
$:a^1\cdots a^Mb^1\cdots b^N:$ and when $s=0$ we do not omit
any factors.
\end{enumerate}
Then
\begin{align*}
:&a^1(z)\cdots a^M(z)::b^1(w)\cdots
b^N(w):= \\
  &\sum_{s=0}^{\min(M,N)}\sum_{i_1<\cdots<i_s,\,
j_1\neq \cdots \neq j_s}\lfloor a^{i_1}b^{j_1}\rfloor\cdots
\lfloor a^{i_s}b^{j_s}\rfloor
:a^1(z)\cdots a^M(z)b^1(w)\cdots
b^N(w):_{(i_1,\dots,i_s;j_1,\dots,j_s)}.
\end{align*}
\end{theorem}

Setting $m=i-\frac{1}{2}$, $i\in\mathbb Z+\frac{1}{2}$ and $x\in\mathfrak g$, define $x_{m+\frac{1}{2}}=x\otimes t^{i-\frac{1}{2}}u=x^1_m$ and $x_m:=x\otimes t^m$.  Define
\begin{align*}
x^1(z)
&:=\sum_{m\in\mathbb Z}x_{m+\frac{1}{2}}z^{-m-1},\quad x(z):=\sum_{m\in\mathbb Z}x_{m}z^{-m-1}.
\end{align*}

The relations in \thmref{3ptthm} then can be rewritten succinctly  as
\begin{align}
[x(z),y(w)]
&= [xy](w)\delta(z/w)-(x,y)\omega_0\partial_w\delta(z/w), \label{r1} \\ 
[x^1(z),y^1(w)]
&= P(w)\left([x,y](w)\delta(z/w) -(x,y)\omega_0\partial_w\delta(z/w)\right)-\frac{1}{2}(x,y)(\partial P(w))\omega_0  \delta(z/w),  \label{r2} \\ 
[x(z),y^1(w)]
&=[x,y]^1(w)\delta(z/w) -(x,y)(\partial_w\psi(c,w)\delta(z/w)-w\psi(c,w)\partial_w\delta(z/w)) \label{r3}\\
&=[x^1(z),y(w)],\notag  
 \end{align}
where $x,y\in\{e,f,h\}$, $P(w)=w^4-2cw^2+1$ and $\psi(c,w)=\sum_{n\in\mathbb Z}\psi_n(c)w^n$ for $\psi_{i+j}(c):=\psi'_{ij}(c)$.

 \section{Oscillator algebras}
 \subsection{The $\beta-\gamma$ system}   The $\beta-\gamma$ is the infinite dimensional oscillator algebra $\hat{\mathfrak a}$ with generators $a_n,a_n^*,a^1_n,a^{1*}_n,\,n\in\mathbb Z$ together with $\mathbf 1$ satisfying the relations 
\begin{gather*}
[a_n,a_m]=[a_m,a_n^1]=[a_m,a_n^{1*}]=[a^*_n,a^*_m]=[a^*_n,a^1_m]=[a^*_n,a^{1*}_{m}]=0,\\
[a_n^{1},a_m^{1}]=[a_n^{1*},a_m^{1*}]=0=[\mathfrak a,\mathbf 1], \\
[a_n,a_m^*]=\delta_{m+n,0}\mathbf 1=[a^1_n,a_m^{1*}].
\end{gather*}
For  $c=a,a^1$ and respectively $X=x,x^1$ with $r=0$ or $r=1$, sets $\mathbb C[\mathbf x]:= \mathbb C[x_n,x_n^1\,|\,n\in\mathbb Z$ and define $\rho:\hat{\mathfrak a}\to \mathfrak{gl}(\mathbb C[\mathbf x])$ by
\begin{align}
\rho_r( c_{m}):&=\begin{cases}
  \partial/\partial
X_{m}&\quad \text{if}\quad m\geq 0,\enspace\text{and}\enspace  r=0
\\ X_{m} &\quad \text{otherwise},
\end{cases}\label{c}
 \\
\rho_r(c_{m}^*):&=
\begin{cases}X_{-m} &\enspace \text{if}\quad m\leq
0,\enspace\text{and}\enspace r=0\\ -\partial/\partial
X_{-m}&\enspace \text{otherwise}. \end{cases}\label{c*}
\end{align}
and $\rho_r(\mathbf 1)=1$.
These two representations can be constructed using induction:
For $r=0$ the representation 
$\rho_0$ is the
$\hat{\mathfrak a}$-module generated by $1=:|0\rangle$, where
$$
a_{m}|0\rangle=a^1_{m}|0\rangle=0,\quad m\geq  0,
\quad a_{m}^*|0\rangle= a_{m}^{1*}|0\rangle=0,\quad m>0.
$$
For $r=1$ the representation 
$\rho_1$ is the
$\hat{\mathfrak a}$-module generated by $1=:|0\rangle$, where
$$
\quad a_{m}^*|0\rangle= a_{m}^{1*}|0\rangle=0,\quad m\in\mathbb Z.
$$
If we write 
$$
 \alpha(z):=\sum_{n\in\mathbb Z}a_nz^{-n-1},\quad  \alpha^*(z):=\sum_{n\in\mathbb Z}a_n^*z^{-n},
$$
and
$$
 \alpha^1(z):=\sum_{n\in\mathbb Z}a^1_nz^{-n-1},\quad  \alpha^{1*}(z):=\sum_{n\in\mathbb Z}a^{1*}_nz^{-n},
$$
then 
\begin{align*}
[\alpha(z),\alpha(w)]&=[\alpha^*(z),\alpha^*(w)]=[\alpha^{1}(z),\alpha^{1}(w)]=[\alpha^{1*}(z),\alpha^{1*}(w)]=0  \\
[\alpha(z),\alpha^*(w)]&=[\alpha^1(z),\alpha^{1*}(w)]
    =\mathbf 1\delta(z/w).
\end{align*}
Corresponding to these two representations there are two possible normal orderings:  For $r=0$ we use the usual normal ordering given by \eqnref{usualnormalordering} and for $r=1$ we define the {\it natural normal ordering} to be 
\begin{alignat*}{2}
\alpha(z)_+&=\alpha(z),\quad &\alpha(z)_-&=0 \\
\alpha^1(z)_+&=\alpha^1(z),\quad &\alpha^1(z)_-&=0 \\
\alpha^*(z)_+&=0,\quad &\alpha^*(z)_-&=\alpha^*(z), \\
\alpha^{1*}(z)_+&=0,\quad &\alpha^{1*}(z)_-&=\alpha^{1*}(z) ,
\end{alignat*}

This means in particular that for $r=0$ we get 
\begin{align}
\lfloor \alpha(z)\alpha^*(w)\rfloor
&=\sum_{m\geq 0} \delta_{m+n,0}z^{-m-1}w^{-n}
=\delta_-(z/w)
=\ 
\,\iota_{z,w}\left(\frac{1}{z-w}\right)\\
\lfloor \alpha^*(z)\alpha(w)\rfloor
&
=-\sum_{m\geq 1} \delta_{m+n,0}z^{-m}
w^{-n-1}
=-\delta_+(w/z)=\,\iota_{z,w}\left(\frac{1}{w-z}
\right)
\end{align}
(where $\iota_{z,w}$ Taylor series expansion in the `` region'' $|z|>|w|$), 
and for $r=1$ 
\begin{align}
\lfloor \alpha\alpha^*\rfloor
&=[\alpha(z)_-,\alpha^*(w)]=0 \\
\lfloor \alpha^*\alpha\rfloor
&=[\alpha^*(z)_-,\alpha(w)]=
-\sum_{\in\mathbb Z} \delta_{m+n,0}z^{-m}
w^{-n-1}
=- \delta(w/z),
\end{align}
where similar results hold for $\alpha^1$.
Notice that in both cases we have
$$
[\alpha(z),\alpha^*(w)]=
\lfloor \alpha(z)\alpha^*(w)\rfloor-\lfloor\alpha^*(w) \alpha(z)\rfloor=\delta(z/w).
$$

The following two Theorems are needed for the proof of our main result:
\begin{theorem}[Taylor's Theorem, \cite{MR99f:17033}, 2.4.3]
\label{Taylorsthm}  Let
$a(z)$ be a formal distribution.  Then in the region $|z-w|<|w|$,
\begin{equation}
a(z)=\sum_{j=0}^\infty \partial_w^{(j)}a(w)(z-w)^j.
\end{equation}
\end{theorem}

\begin{theorem}[\cite{MR99f:17033}, Theorem 2.3.2]\label{kac}  Set $\mathbb C[\mathbf x]=\mathbb C[x_n,x^1_n|n\in\mathbb Z]$ and $\mathbb C[\mathbf y]= C[y_m,y_m^1|m\in\mathbb N^*]$.  Let $a(z)$ and $b(z)$ 
be formal distributions with coefficients in the associative algebra 
 $\End(\mathbb C[\mathbf x]\otimes \mathbb C[\mathbf y])$ where we are using the usual normal ordering.   The
following are equivalent
\begin{enumerate}[(i)]
\item
$\displaystyle{[a(z),b(w)]=\sum_{j=0}^{N-1}\partial_w^{(j)}
\delta(z-w)c^j(w)}$, where $c^j(w)\in \End(\mathbb C[\mathbf x]\otimes \mathbb 
C[\mathbf y])[\![w,w^{-1}]\!]$.
\item
$\displaystyle{\lfloor
ab\rfloor=\sum_{j=0}^{N-1}\iota_{z,w}\left(\frac{1}{(z-w)^{j+1}}
\right)
c^j(w)}$.
\end{enumerate}\label{Kacsthm}
\end{theorem}

So the singular part of the {\it operator product
expansion}
$$
\lfloor
ab\rfloor=\sum_{j=0}^{N-1}\iota_{z,w}\left(\frac{1}{(z-w)^{j+1}}
\right)c^j(w)
$$
completely determines the bracket of mutually local formal
distributions $a(z)$ and $b(w)$.   In the physics literature one writes
$$
a(z)b(w)\sim \sum_{j=0}^{N-1}\frac{c^j(w)}{(z-w)^{j+1}}.
$$

\subsection{DJKM Heisenberg algebra}Set
\begin{equation}
\psi'_{ij}(c)=\begin{cases} 
\mathbf 1_{i+j-2}&\quad \text{for}\quad i+j=1,0,-1,-2 \\
P_{-3,i+j-2}(c) (\mathbf 1_{-3}+c\mathbf 1_{-1})&\quad \text{for} \quad i+j =2n-1\geq 3, \\
P_{-3,i+j-2}(c) (c\mathbf 1_{-3}+\mathbf 1_{-1})&\quad \text{for} \quad i+j =-2n+1\leq - 3,   \\
P_{-4,|i+j|-2}(c) \mathbf 1_{-4} +P_{-2,|i+j|-2}(c)\mathbf 1_{-2}&\quad\text{for}\quad |i+j| =2n \geq 2,   \\
\end{cases}
\end{equation}
for $n\in\mathbb Z$ and $\psi'(c,w)=\sum_{n\in\mathbb Z}\psi'_n(c)w^n$ for $\psi'_{i+j}(c):=\psi'_{ij}(c)$.

The Cartan subalgebra $\mathfrak h$ tensored with $\mathcal R$ generates a subalgebra of $\hat{{\mathfrak g}}$ which is an extension of an oscillator algebra.    This extension motivates the following definition:  The Lie algebra with generators $b_{m},b_m^1$, $m\in\mathbb Z$, $\mathbf 1_i$, $i\in\{0,-1,-2,-3,-4\}$ and relations
\begin{align}
[b_{m},b_{n}]&=(n-m)\,\delta_{m+n,0}\mathbf 1_0=-2m\,\delta_{m+n,0}\mathbf 1_0\label{b1} \\
[b^1_m,b_n^1] &=2 \left((n+2)\delta_{m+n,-4}-2c(n+1)\delta_{m+n,-2}+n\delta_{m+n,0}\right)\mathbf 1_0\label{b2}\\
[b^1_m,b_n] &=2n\psi'_{mn}(c)\label{b3} \\
[b_{m},\mathbf 1_i]&=[b_{m}^1,\mathbf 1_i]= 0.\label{b4}
\end{align}
we will give the appellation the {\it DJKM (affine) Heisenberg algebra} and denote it by $\hat{\mathfrak b}_3$.

If we introduce the formal distributions
\begin{equation} 
\beta(z):=\sum_{n\in\mathbb Z} b_nz^{-n-1},\quad \beta^1(z):=\sum_{n\in\mathbb Z}b_n^1z^{-n-1}=\sum_{n \in\mathbb Z}b_{n+\frac{1}{2}}z^{-n-1}.
\end{equation}
(where $b_{n+\frac{1}{2}}:=b^1_n$)
then using calculations done earlier for DJKM Lie algebra we can see that the relations above can be rewritten in the form
\begin{align*}\label{bosonrelations}
[\beta(z),\beta(w)]&=2\mathbf 1_0\partial_z\delta(z/w)=-2\partial_w\delta(z/w)\mathbf 1_0 \\
[\beta^1(z),\beta^1(w)]
&=-2\left(P(w)\partial_w\delta(z/w)+ \frac{1}{2}\partial_wP(w)\delta(z/w)\right)\mathbf 1_0 \\
[\beta^1(z),\beta(w)]&=2\partial_w\psi'(c,w)\delta(z/w)-2w\psi'(c,w)\partial_w\delta(z/w).
\end{align*}

 Set
\begin{gather*}
\hat{\mathfrak h}_3^\pm:=\sum_{n\gtrless 0}\left(\mathbb Cb_n+\mathbb Cb_n^1\right),\\ 
\hat{ \mathfrak h}_3^0:=  \mathbb C\mathbf 1_0\oplus \mathbb C\mathbf 1_{-1}\oplus \mathbb C\mathbf 1_{-2}\oplus \mathbb C\mathbf 1_{-3}\oplus \mathbb C\mathbf 1_{-4}\oplus \mathbb Cb_0\oplus \mathbb Cb^1_0.
\end{gather*}
We introduce a Borel type subalgebra
\begin{align*}
\hat{\mathfrak b}_3&= \hat{\mathfrak h}_3^+\oplus \hat{\mathfrak h}_3^0.
\end{align*}
Due to the defining relations above one can see that $\hat{\mathfrak b}_3$ is a subalgebra.

\begin{lem}\label{heisenbergprop}
Let $\mathcal V=\mathbb C\mathbf v_0\oplus \mathbb C\mathbf v_1$ be a two dimensional representation of $\hat{\mathfrak h}_3^+\mathbf v_i=0$ for $i=0,1$.   Suppose  $\lambda,\mu,\nu,\varkappa, \chi_{-1}, \chi_{-2}, \chi_{-3}, \chi_{-4},\kappa_0 \in \mathbb C$  are such that
\begin{align*}
b_0\mathbf v_0&=\lambda \mathbf v_0,  &b_0\mathbf v_1&=\lambda \mathbf v_1 \\
b_0^1\mathbf v_0&=\mu \mathbf v_0+\nu \mathbf v_1,  &b_0^1\mathbf v_1&=\varkappa \mathbf v_0+\mu \mathbf v_1\\
\mathbf 1_j\mathbf v_i&=\chi _j  \mathbf v_i,\quad   &\mathbf 1_0\mathbf v_i&=\kappa _0\mathbf v_i,\quad i=0,1,\quad j=-1,-2,-3,-4.
\end{align*}
Then the above defines a representation of  $\hat{\mathfrak b}_3$.  Not only that but also $\chi_{-1}=\chi_{-2}=\chi_{-3}=\chi_{-4}=0$ and $\psi'_{mn}=0$, for all $m,n\in\mathbb Z$.

\end{lem}
\begin{proof} Since $b_m$ acts by scalar multiplication for $m,n\geq 0$, the first defining relation \eqnref{b1} is satisfied for $m,n\geq 0$. 
The second relation \eqnref{b2} is also satisfied as the right hand side is zero if $m\geq 0,n\geq 0$.  If $n=0$, then since $b_0$ acts by a scalar, the relation \eqnref{b3} leads to no condition on $\lambda,\mu,\nu,\varkappa, \chi_1,\kappa_0 \in \mathfrak h_3^0$.    If $m\neq0$ and $n\neq0$, the third relation give us
$$
0=b^1_mb_n\mathbf v_i-b_nb^1_m\mathbf v_i=[b^1_m,b_n] \mathbf v_i=2n\psi'_{mn}\mathbf v_i=0.
$$
and then $\psi'_{mn}=0$ for $n\neq0$. Consequently $\chi_{-1}=\chi_{-2}=\chi_{-3}=\chi_{-4}=0$ and $\psi'_{mn}=0$ for all $m,n\in\mathbb Z$.
\end{proof}

\begin{lem} \label{rhorep}The linear map $\rho:\hat{\mathfrak b}_3\to \text{End}(\mathbb C[\mathbf y]\otimes \mathcal V)$ defined  by 
\begin{align}
\rho(b_{n})&=y_{n} \quad \text{ for }n<0 \\
\rho(b_{n}^1)&=y_{n}^1+\delta_{n,-1}\partial_{y_{-3}^1}\chi_0-\delta_{n,-3}\partial_{y_{-1}^1}\chi_0\quad \text{ for }n<0 \\
\rho (b_n) &= -2n \partial_{ y_{-n} }\chi_0 \quad \text{ for }n>0 \\
\rho(b^1_n)&=2(n+2) \partial_{y^1_{-n-4}} \chi_0-4c(n+1) \partial_{y^1_{-n-2}} \chi_0 +2n \partial_{y^1_{-n}} \chi_0  \quad \text{ for }n>0\\
\rho(b_{0})&=\lambda \\
\rho(b_{0}^1)&=4\partial_{y_{-4}^1}\chi_0-2c\partial_{y_{-2}^1}\chi_0 +B_0^1.
\end{align}
is a representation of $\hat{\mathfrak b}_3$.
\end{lem}
\begin{proof}
For $m,n> 0$, it is straight forward to see $[\rho(b_n),\rho(b_m)]=[\rho(b^1_n),\rho(b^1_m)]=0$, and similarly for $m,n<0$, $[\rho(b_n),\rho(b_m)]=0$ and $[\rho(b^1_n),\rho(b^1_m)]=0$ if $n\notin\{-1,-3\}$ and
\begin{align*}
[\rho(b^1_{-1}),\rho(b^1_m)]&=[y_{-1}^1+\partial_{y_{-3}^1}\chi_0,y_{m}^1+\delta_{m,-1}\partial_{y_{-3}^1}\chi_0-\delta_{m,-3}\partial_{y_{-1}^1}\chi_0]\\
&=-\delta_{m,-3}\chi_0[y_{-1}^1,\partial_{y_{-1}^1}]\chi_0+\delta_{m,-3}[\partial_{y_{-3}^1},y_{-3}^1]\chi_0\\
&=-2\delta_{m,-3}\chi_0,
\end{align*}
\begin{align*}
[\rho(b^1_{-3}),\rho(b^1_m)]&=[y_{-3}^1-\partial_{y_{-1}^1}\chi_0,y_{m}^1+\delta_{m,-1}\partial_{y_{-3}^1}\chi_0-\delta_{m,-3}\partial_{y_{-1}^1}\chi_0]\\
&=\delta_{m,-1}\chi_0[y_{-3}^1,\partial_{y_{-3}^1}]\chi_0-\delta_{m,-1}[\partial_{y_{-1}^1},y_{-1}^1]\chi_0\\
&=2\delta_{m,-1}\chi_0,
\end{align*}
\begin{align*}
[\rho(b^1_{0}),\rho(b^1_m)]&=[4\partial_{y_{-4}^1}\chi_0-2c\partial_{y_{-2}^1}\chi_0,y_{m}^1+\delta_{m,-1}\partial_{y_{-3}^1}\chi_0-\delta_{m,-3}\partial_{y_{-1}^1}\chi_0]\\
&=-4\delta_{m,-4}\chi_0+2c\delta_{m,-2}\chi_0.
\end{align*}

For $m>0$ and $n\leq 0$ we have 
\begin{align*}
[\rho(b_m),\rho(b_n)]&=[-2m \partial_{ y_{-m} }\chi_0 ,y_n]=-2m\delta_{m,-n}\chi_0,  \\ \\ 
[\rho(b_m),\rho(b^1_n)]&=[-2m \partial_{ y_{-m} }\chi_0  ,y^1_n+\delta_{n,-1}\partial_{y_{-3}^1}\chi_0-\delta_{n,-3}\partial_{y_{-1}^1}\chi_0]=0 ,    \\  \\
[\rho(b^1_m),\rho(b^1_n)]&=[2(n+2) \partial_{y^1_{-n-4}} \chi_0-4c(n+1) \partial_{y^1_{-n-2}} \chi_0 +2n \partial_{y^1_{-n}} \chi_0\\
&\hspace{1.0cm} ,y^1_n+\delta_{n,-1}\partial_{y_{-3}^1}\chi_0-\delta_{n,-3}\partial_{y_{-1}^1}\chi_0]  \\ 
&= 2(n+2) \delta_{m+n,-4} \chi_0 -4c(n+1) \delta_{m+n,-2} \chi_0+2n\delta_{m+n,0}\chi_0,      \\  \\
[\rho(b^1_m),\rho(b_n)]&=[2(m+2) \partial_{y^1_{-m-4}} \chi_0-4c(m+1) \partial_{y^1_{-m-2}} \chi_0 +2m \partial_{y^1_{-m}} \chi_0, \\
&\hspace{1.0cm} y_n+\delta_{n,-1}\partial_{y_{-3}^1}\chi_0-\delta_{n,-3}\partial_{y_{-1}^1}\chi_0]  \\ 
&=0.
\end{align*}
\end{proof}

\section{Two realizations of DJKM algebra $\hat{{\mathfrak g}}$.}
Recall
\begin{align}
P(w) =w^4-2cw^2+1.
\end{align}

Our main result is the following 
\begin{theorem} \label{mainthm}  Fix $r\in\{0,1\}$, which then fixes the corresponding normal ordering convention defined in the previous section.  Set $\hat{{\mathfrak g}} =\left(\mathfrak{sl}(2,\mathbb C)\otimes \mathcal R\right)\oplus \mathbb C\omega_0\oplus \mathbb C\omega_{-1}\oplus \mathbb C\omega_{-2}\oplus \mathbb C\omega_{-3}\oplus \mathbb C\omega_{-4}$ and assume that $\chi_0\in\mathbb C$ and $\mathcal V$ as in \lemref{heisenbergprop}.   Then using \eqnref{c}, \eqnref{c*} and \lemref{rhorep}, the following defines a representation of DJKM algebra ${\mathfrak g}$ on $\mathbb C[\mathbf x]\otimes \mathbb C[\mathbf y]\otimes \mathcal V$:
\begin{align*}
\tau(\omega_{-1})&=\tau(\omega_{-2})=\tau(\omega_{-3})=\tau(\omega_{-4})=0, \qquad
\tau(\omega_0)=\chi_0=\kappa_0+4\delta_{r,0} ,  \\ 
\tau(f(z))&=-\alpha, \qquad
\tau(f^1(z))=- \alpha^1,   \\ \\
\tau(h(z))
&=2\left(:\alpha\alpha^*:+:\alpha^1\alpha^{1*}: \right)
     +\beta , \\  \\
\tau(h^1(z))
&=2\left(:\alpha^1\alpha^*: +P(z):\alpha\alpha^{1*}: \right) +\beta^1,  \\  \\
\tau(e(z)) 
&=:\alpha(\alpha^*)^2:+P(z):\alpha(\alpha^{1*})^2: +2 :\alpha^1\alpha^*\alpha^{1*}:+\beta\alpha^*+\beta^1\alpha^{1*} +\chi_0\partial\alpha^*  \\ \\
\tau(e^1(z))  
&=\alpha^1\alpha^*\alpha^* 
 	+P(z)\left(\alpha^1 (\alpha^{1*} )^2 +2 : \alpha \alpha^{*} \alpha^{1*}:\right)  \\
&\quad +\beta^1 \alpha^* +P(z)\beta \alpha^{1*} +\chi_0\left(P(z)   \partial_z\alpha^{1*}   +\dfrac{1}{2}\partial_zP(z)   \alpha^{1*} \right) .
\end{align*}
\end{theorem}

\begin{proof} The proof is very similar to the proof of Theorem 5.1 in \cite{MR2541818} and Theorem 5.1 in \cite{CoxJur}.
We  need to check that the following table is preserved under  $\tau$. 
\begin{table}[htdp]
\caption{}
\begin{center}
\begin{tabular}{c|cccccc}
$[\cdot_\lambda \cdot]$ & $f(w)$ & $f^1(w)$ & $h(w)$ & $h^1(w)$ & $e(w)$ & $e^1(w)$ \\
\hline
$f(z)$ & $0$ &  $0$ & $*$ &  $*$ &$*$  &$*$    \\
$f^1(z)$ &   &  $0$  & $*$  & $*$  & $*$  &$*$  \\
$h(z)$ &  &   &  $*$  & $*$  & $*$  & $*$ \\
$h^1(z)$ &&   &     & $*$  &    $*$ & $*$ \\
$e(z)$  &&   &     &  &    $0$ &  $0$  \\
$e^1(z)$ & &   &     &  &   &  $0$ \\
\end{tabular}
\end{center}
\label{default}
\end{table}%

Here $*$ indicates nonzero formal distributions that are obtained from the the defining relations \eqref{r1}, \eqref{r2}, and \eqref{r3}.
The proof is carried out using Wick's Theorem and Taylor's Theorem. 
We are going to make use of V. Kac's $\lambda$-notation (see \cite{MR99f:17033} section 2.2  for some of its properties) used in operator product expansions.  If $a(z)$ and $b(w)$ are formal distributions, then
$$
[a(z),b(w)]=\sum_{j=0}^\infty \frac{(a_{(j)}b)(w)}{(z-w)^{j+1}}
$$
is transformed under the {\it formal Fourier transform} 
$$
F^{\lambda}_{z,w}a(z,w)=\text{Res}_ze^{\lambda(z-w)}a(z,w),
$$
 into the sum
\begin{equation*}
[a_\lambda b]=\sum_{j=0}^\infty \frac{\lambda^j}{j!}a_{(j)}b.
\end{equation*}

\begin{align*}
[\tau(f)_\lambda\tau(f)]&=0,\quad [\tau(f)_\lambda\tau(f^1)]= 0,\quad [\tau(f^1)_\lambda\tau(f^1)]=0 \\ \\
[\tau(f)_\lambda\tau(h)]&=-\Big[\alpha_\lambda\Big(2\left(\alpha\alpha^*+\alpha^1\alpha^{1*} \right)   +\beta \Big)\Big]=  -2\alpha =   2\tau(f) , \\  \\
[\tau(f)_\lambda\tau(h^1)]&=-[\alpha_\lambda\left(2\left(\alpha^1\alpha^*
   +P\alpha\alpha^{1*} \right)
     +\beta^1\right)  ]= -2\alpha^{1}=   2\tau(f^1), \\  \\
[\tau(f)_\lambda\tau(e)]&=- [\alpha_\lambda\left(:\alpha(\alpha^*)^2:+P:\alpha(\alpha^{1*})^2: +2 :\alpha^1\alpha^*\alpha^{1*}:+\beta\alpha^*+\beta^1\alpha^{1*} +\chi_0\partial\alpha^*\right)] \\  
&=-2\left(:\alpha\alpha^*:    + :\alpha^1\alpha^{1*}: \right)-\beta-\chi_0\lambda  
= -\tau(h)-\chi_0\lambda   \\ \\
\tau(f)_\lambda\tau(e^1)]&=- \Big[\alpha_\lambda\Big(\alpha^1(\alpha^*)^2 
 	+P\left(\alpha^1 (\alpha^{1*} )^2 +2 : \alpha \alpha^{*} \alpha^{1*}:\right)  \  +\beta^1 \alpha^* +P\beta \alpha^{1*} +\chi_0\left(P  \partial \alpha^{1*}   +\frac{1}{2}\partial  P   \alpha^{1*} \right) \Big)\Big]\\  
&=-2\left(:\alpha^1\alpha^*:    + P:\alpha\alpha^{1*}: \right)-\beta^1   =-\tau(h^1).
\\ \\
[\tau(f^1)_\lambda\tau(h)]&=-[\alpha^1_\lambda\left(2\left(:\alpha\alpha^*:+:\alpha^1\alpha^{1*}: \right)   +\beta\right) ]=  -2\alpha^1 =   2\tau(f^1),  \\  \\
[\tau(f^1)_\lambda\tau(h^1)]&=-[\alpha^1_\lambda\left(2\left(:\alpha^1\alpha^*:
   +P:\alpha\alpha^{1*}: \right)
     +\beta^1\right)] =-2P\alpha^{1} =2P\tau(f^{1}) , \\  \\
[\tau(f^1)_\lambda\tau(e)]&=- [\alpha^1_\lambda\left(:\alpha(\alpha^*)^2:+P:\alpha(\alpha^{1*})^2: +2 :\alpha^1\alpha^*\alpha^{1*}:+\beta\alpha^*+\beta^1\alpha^{1*} +\chi_0\partial\alpha^*\right)] \\  
&=- \Big(  2P:\alpha\alpha^{1*}: +2 :\alpha^1\alpha^*: +\beta^1  \Big)  
=-\tau(h^1) \\ \\
[\tau(f^1)_\lambda\tau(e^1)]&=-[ \alpha^1_\lambda\Big(\alpha^1(\alpha^*)^2 
 	+P\left(\alpha^1 (\alpha^{1*} )^2 +2 : \alpha \alpha^{*} \alpha^{1*}:\right)  \  +\beta^1 \alpha^* +P\beta \alpha^{1*} +\chi_0\left(P  \partial \alpha^{1*}   +\frac{1}{2}(\partial  P)   \alpha^{1*} \right) \Big)] \\  
&=-\Big(
 	P\left(  2\left(:\alpha^1\alpha^{1*}:   + :\alpha\alpha^* :\right)
	+\beta+\chi_0\lambda  \right)+\frac{1}{2}\chi_0\partial P  \Big)  \\
&=-\left(P\tau(h)+P\chi_0\lambda+\chi_0\frac{1}{2}\partial P\right).
\end{align*}
 Note that $:a(z)b(z):$ and $:b(z)a(z):$ are usually not equal, but $:\alpha^1(w)\alpha^{*1}(w):=:\alpha^{1*}(w)\alpha^1(w):$ and $:\alpha(w)\alpha^*(w):=:\alpha^*(w)\alpha (w):$.  Thus we calculate

\begin{align*}
[\tau(h)_\lambda\tau(h)]&=\left(2\left(:\alpha\alpha^*:+:\alpha^1\alpha^{1*}: \right)   +\beta\right)  _\lambda \left(2\left(:\alpha\alpha^*:+:\alpha^1\alpha^{1*}: \right)   +\beta\right)]  \\
&= 4\Big(-:\alpha\alpha^*: +:\alpha^*\alpha: -:\alpha^1\alpha^{1*}: +:\alpha^{1*}\alpha^1:  \Big)  
-8\delta_{r,0}\lambda +[\beta_\lambda\beta]   \\
&=-2(4\delta_{r,0}+\kappa_0)\lambda
 \end{align*}
 Which can be put into the form of \eqref{r1}:
 \begin{align*}
\left[\tau(h(z)),\tau(h(w))\right]&=
 -2(4\delta_{r,0}+\kappa_0)\partial_{w}\delta(z/w)=
 -2\chi_0\partial_{w}\delta(z/w)=
 \tau\left(-2\omega_0\partial_{w}\delta(z/w)\right).  
  \end{align*}
  
Next we calculate

\begin{align*}
[\tau(h)_\lambda\tau(h^1)]&=\left[\left(2\left(:\alpha\alpha^*:+:\alpha^1\alpha^{1*}: \right)   +\beta \right) _\lambda \left(2\left(:\alpha^1\alpha^*:
   +P:\alpha\alpha^{1*}: \right)+\beta^1  \right) \right]\\  
&=4\Big(\left(:\alpha^* \alpha^1 : -:\alpha^1 \alpha^{*} :\right)   + P \left(-:\alpha\alpha^{1*} :  +:\alpha^{1*}\alpha :  \right)    \Big)  +[\beta_\lambda\beta^1]   
\end{align*}
Since $[a_n,a_m^{1*}]=[a^1_n,a_m^{*}]=0$, we have 
 \begin{align*}
\left[\tau(h(z)),\tau(h^1(w))\right]&=[\beta(z),\beta^1(w)]=0.
 \end{align*}
 As $\tau(\omega_1) =0$,  relation \eqref{r3} is satisfied. 
 
 We continue with

\begin{align*}
[\tau(h^1)_\lambda\tau(h^1)]&=\left[2\left(:\alpha^1\alpha^*:
   +P:\alpha\alpha^{1*}: \right)
     +\beta^1  \right)_\lambda \left(2\left(:\alpha^1 \alpha^* :
   +P:\alpha\alpha^{1*}: \right)
     +\beta^1 )  \right] \\  
&= 4P\left(-:\alpha\alpha^*:+:\alpha^{1*}\alpha^1:\right) +4P\left(-:\alpha^1\alpha^{1*} : +:\alpha^*\alpha: \right) 
  \\
&\quad  -8\delta_{r,0}P\lambda   -4\delta_{r,0}\partial P    + [\beta^1_\lambda \beta^1]   \\
&=-8\delta_{r,0}P\lambda   -4\delta_{r,0}\partial P   -2\kappa_0(P\lambda+\frac{1}{2}\partial P).
 \end{align*}
 Yielding the relation
\begin{align*}
\left[\tau(h^1(z)),\tau(h^1(w))\right]&=-2(4\delta_{r,0}+\kappa_0)\left(\left(w^4-2cw^2+1\right)\partial_w\delta(z/w)   +(2w^3-2cw))\delta(z/w)\right) \\
& = \tau(-(h,h)  \omega_0 P \partial_w\delta(z/w) -\frac{1}{2}(h,h) \partial P \omega_0  \delta(z/w))
\end{align*}

Next we calculate the $h$'s paired with the $e$'s:

\begin{align*}
[\tau(h)_\lambda \tau(e )]&=\Big[\Big(2\left(:\alpha\alpha^*:+:\alpha^1\alpha^{1*}: \right)
     +\beta \Big)_\lambda \\
& \quad :\alpha(\alpha^*)^2:+ P:\alpha(\alpha^{1*})^2: +2 :\alpha^1\alpha^*\alpha^{1*}:+ \beta\alpha^*+\beta^1\alpha^{1*} +\chi_0\partial\alpha^*\Big] \\ 
& = 4: \alpha (\alpha^*)^2:  - 2 : \alpha  (\alpha^*)^2:   - 4 \delta_{r,0} \alpha^* \lambda - 2P : \alpha(\alpha^{1*})^2:  
 + 4: \alpha^*  \alpha^1  \alpha^{1*}  : + 2  \alpha^*  \beta 
 + 2\chi_0  \alpha^* \lambda \\
 &\quad+  2\chi_0  \partial  \alpha^*+ 4P : \alpha (\alpha^{1*})^2: - 4 \delta_{r,0} \alpha^* \lambda + 2     \beta^1   \alpha^{1*}  - 2 \lambda \alpha^* \kappa_0 \\
&= 2 \tau (e)
\end{align*}
and
\begin{align*}
[\tau(h^1)_\lambda \tau(e)] & = \Big[ \left(2\left(:\alpha^1\alpha^*:
   +P:\alpha\alpha^{1*}: \right) + \beta^1 \right)_\lambda 
  \\
 &\hskip 40pt :\alpha(\alpha^*)^2:+ P:\alpha(\alpha^{1*})^2: +2 :\alpha^1\alpha^*\alpha^{1*}:+ \beta\alpha^*+\beta^1\alpha^{1*} +\chi_0\partial\alpha^*\Big]\\
 &=  -2  : \alpha^1( \alpha^*)^2 
 + 2P(2 : \alpha \alpha^* \alpha^{1*}: - :  \alpha^1  (\alpha^{1*})^2: - 2 \delta_{r,0} \alpha^{1*}  \lambda ) - 4 \delta_{r,0} \partial P  \alpha^{1*} \\
 &\quad+ 4 :  \alpha^1( \alpha^*)^2: + 2 \alpha^* \beta^1 + 4P: \alpha \alpha^{1*} \alpha^* + 4P(  : \alpha^1 (\alpha^{1*})^2 : - : \alpha \alpha^{*}\alpha^{1*}:  -\delta_{r,0}   \alpha^{1*}\lambda ) \\
&\quad+ 2P  \beta \alpha^{1*}+ 2 \chi_0 ( P\partial \alpha^{1*}   + \partial P \alpha^{1*}   +P \alpha^{1*}  
 \lambda  )- 2 ( P \lambda + \frac{1}{2} \partial P ) \alpha^{1*} \kappa_0\\
&=  2   :  \alpha^1( \alpha^*)^2:  +  2P  :  \alpha^1  (\alpha^{1*})^2:  + 4P:  \alpha \alpha^*  \alpha^{1*}:  +
2 \delta(z/w) \alpha^* \beta^1 \\
&\quad + 2P\beta \alpha^{1*}+ 2P \chi_0 \partial   \alpha^{1*}  +  \partial P  \alpha^{1*} \chi_0 \\
& = 2\tau (e^1  )
\end{align*}

Next we must calculate

\begin{align*}
[\tau(h)_\lambda \tau(e^1)]&=\Big[\Big(2\left(:\alpha\alpha^*:+:\alpha^1\alpha^{1*}: \right)
     +\beta \Big)_\lambda \\
&\quad  : \alpha^1\alpha^*\alpha^* :
 	+P\left(:\alpha^1 (\alpha^{1*} )^2 :+2 : \alpha \alpha^{*} \alpha^{1*}:\right)  +\beta^1 \alpha^* +P\beta \alpha^{1*}  +\chi_0\left(P  \partial \alpha^{1*}   +(1/2 \partial P )  \alpha^{1*} \right)\Big] \\  
&= 4  :\alpha^1 (\alpha^* )^2: - 4 P\delta_{r,0}   \alpha^{1*} \lambda + 2   \alpha^* \beta^1  - 2\delta(z/w) : \alpha^1 \alpha^* \alpha^* :-4P\delta_{r,0}  \alpha^{1*} \lambda + 2P  : \alpha^{1*}\alpha^1 \alpha^{1*}: \\
&\quad+ 4P : \alpha^{1*} \alpha \alpha^*:+2P\beta\alpha^{1*} + 2 \chi_0(   P \alpha^{1*}  \lambda +  P  \partial \alpha^{1*} + \frac{1} {2} \partial P   \alpha^{1*} )- 2 P \alpha^{1*}  \kappa_0 \lambda \\
& = 2:  \alpha^1( \alpha^*)^2:   +  2P : \alpha^1 (\alpha^{1*})^2:  + 4P: \alpha \alpha^* \alpha^{1*} : \\
 &\quad + 2 \beta^1  \alpha^*   +2P\beta\alpha^{1*}     +2\chi_0\left(P  \partial_w\alpha^{1*}   +(1/2 \partial P )  \alpha^{1*} \right)\\
 & =  2\tau(e^1 )
\end{align*}
and the proof for $[\tau(h^1)_\lambda \tau(e^1)]$ is similar. 
%

We prove the Serre relation for just one of the relations, $[\tau(e)_\lambda \tau(e^1)]$, and the proof of the others ($[\tau(e)_\lambda \tau(e)]$, $[\tau(e^1)_\lambda \tau(e^1)]$) are similar as the reader can verify. 

\begin{align*} 
[\tau(e)_\lambda \tau(e^1)]
&=  \Big[:\alpha(\alpha^*)^2:  _\lambda \Big(\alpha^1(\alpha^*)^2 
 	+(w^4-2cw^2+1)\left(\alpha^1 (\alpha^{1*} )^2 +2 : \alpha \alpha^{*} \alpha^{1*}:\right)  \\
&\qquad +\beta^1 \alpha^* +(w^4-2cw^2+1)\beta \alpha^{1*}  +\chi_0\left((w^4-2cw^2+1)   \partial_w\alpha^{1*}   +(2w^3-2cw)   \alpha^{1*} \right)\Big)   \Big]\\  
&\quad + \Big[P:\alpha(\alpha^{1*})^2 :{_\lambda}  \Big(:\alpha^1(\alpha^*)^2 :
 	+(w^4-2cw^2+1)\left(\alpha^1 (\alpha^{1*} )^2 +2 : \alpha \alpha^{*} \alpha^{1*}:\right)  \\
&\qquad +\beta^1 \alpha^* +(w^4-2cw^2+1)\beta \alpha^{1*}  +\chi_0\left((w^4-2cw^2+1)   \partial_w\alpha^{1*}   +(2w^3-2cw)   \alpha^{1*} \right)\Big)   \Big]\\   
&\quad +   \Big[2 :\alpha^1\alpha^*\alpha^{1*}:{_\lambda}  \Big(:\alpha^1(\alpha^*)^2 :
 	+(w^4-2cw^2+1)\left(\alpha^1 (\alpha^{1*} )^2 +2 : \alpha \alpha^{*} \alpha^{1*}:\right)  \\
&\qquad +\beta^1 \alpha^* +(w^4-2cw^2+1)\beta \alpha^{1*}  +\chi_0\left((w^4-2cw^2+1)   \partial_w\alpha^{1*}   +(2w^3-2cw)   \alpha^{1*} \right)\Big)   \Big]\\   
&\quad    +  \Big[\beta\alpha^*{ _\lambda}  \Big(:\alpha^1(\alpha^*)^2 :
 	+(w^4-2cw^2+1)\left(\alpha^1 (\alpha^{1*} )^2 +2 : \alpha \alpha^{*} \alpha^{1*}:\right)  \\
&\qquad +\beta^1 \alpha^* +(w^4-2cw^2+1)\beta \alpha^{1*}  +\chi_0\left((w^4-2cw^2+1)   \partial_w\alpha^{1*}   +(2w^3-2cw)   \alpha^{1*} \right)\Big)    \Big]\\  
&\quad  +\Big[\beta^1\alpha^{1*}{ _\lambda}  \Big(:\alpha^1(\alpha^*)^2 :
 	+(w^4-2cw^2+1)\left(\alpha^1 (\alpha^{1*} )^2 +2 : \alpha \alpha^{*} \alpha^{1*}:\right)  \\
&\qquad +\beta^1 \alpha^* +(w^4-2cw^2+1)\beta \alpha^{1*}  +\chi_0\left((w^4-2cw^2+1)   \partial_w\alpha^{1*}   +(2w^3-2cw)   \alpha^{1*} \right)\Big)    \Big]\\   
&\quad +\Big[  \chi_0\partial\alpha^*{_\lambda}  \Big(:\alpha^1(\alpha^*)^2 :
 	+(w^4-2cw^2+1)\left(\alpha^1 (\alpha^{1*} )^2 +2 : \alpha \alpha^{*} \alpha^{1*}:\right)  \\
&\qquad +\beta^1 \alpha^* +(w^4-2cw^2+1)\beta \alpha^{1*}  +\chi_0\left((w^4-2cw^2+1)   \partial_w\alpha^{1*}   +(2w^3-2cw)   \alpha^{1*} \right)\Big)   \Big] \\ \\  
&=   \Big[:\alpha(\alpha^*)^2:  _\lambda \Big(:\alpha^1(\alpha^*)^2 :
 	+2P : \alpha \alpha^{*} \alpha^{1*}:  +\beta^1 \alpha^* \Big)   \Big]\\  \\ 
&\quad + \Big[P:\alpha(\alpha^{1*})^2 :{_\lambda}  \Big(:\alpha^1(\alpha^*)^2:
 	+P\left(:\alpha^1 (\alpha^{1*} )^2: +2 : \alpha \alpha^{*} \alpha^{1*}:\right) +\beta^1 \alpha^* \Big)   \Big]\\  \\ 
&\quad +   \Big[2 :\alpha^1\alpha^*\alpha^{1*}:{_\lambda}  \Big(\alpha^1
(\alpha^*)^2 
 	+P\left(:\alpha^1 (\alpha^{1*} )^2: +2 : \alpha \alpha^{*} \alpha^{1*}:\right)  +P\beta \alpha^{1*}  \\
&\hskip 100pt +\chi_0\left((w^4-2cw^2+1)   \partial_w\alpha^{1*}   +(2w^3-2cw)   \alpha^{1*} \right)\Big)   \Big]\\  \\ 
&\quad    +  \Big[\beta\alpha^*{ _\lambda}  \Big(2P : \alpha \alpha^{*} \alpha^{1*}: +\beta^1 \alpha^* +P\beta \alpha^{1*}\Big)    \Big]\\  \\
&\quad  +\Big[\beta^1\alpha^{1*}{ _\lambda}  \Big(:\alpha^1(\alpha^*)^2 :
 	+P\alpha^1 (\alpha^{1*} )^2  +\beta^1 \alpha^* +P\beta \alpha^{1*} \Big)    \Big]\\  \\  
&\quad 
	  +\Big[  \chi_0\partial\alpha^*{_\lambda}  \Big( 2P : \alpha \alpha^{*} \alpha^{1*}:\Big)   \Big]  
\end{align*}

\begin{align*}
&=   2:\alpha^1\alpha^*(\alpha^*)^2 :
 	+2P: \alpha( \alpha^{*})^2 \alpha^{1*}: -4P:\alpha( \alpha^{*})^2 \alpha^{1*}: -4\delta_{r,0}P:   \alpha^{*} \alpha^{1*}:\lambda-4\delta_{r,0}P:  \partial( \alpha^{*}) \alpha^{1*}:  \\
	&\hskip 100pt   +\beta^1 (\alpha^*)^2
	 \\  \\ 
&\quad -2P:\alpha(\alpha^*)^2\alpha^{1*}:+2P\alpha^1\alpha^*(\alpha^{1*})^2 \\ 
&\quad  -4\delta_{r,0}  P:\alpha^{1*} \alpha^*:\lambda-4\delta_{r,0}\partial P:\alpha^{1*} \alpha^*:-4\delta_{r,0} P:\partial\alpha^{1*} \alpha^*:
  \\
&\quad -2P^2:\alpha  (\alpha^{1*} )^3: +2 P^2: \alpha(\alpha^{1*})^3:  +P\beta^1 (\alpha^{1*})^2  
 \\  \\ 
&\quad -  2 :\alpha^1\alpha^*(\alpha^*)^2:
 	+ 4P:\alpha^1  \alpha^*(\alpha^{1*})^2:-2P:\alpha^1\alpha^* (\alpha^{1*} )^2:-4\delta_{r,0}P: \alpha^*\alpha^{1*}:\lambda  -4\delta_{r,0}P: \partial(\alpha^*)\alpha^{1*}: \\
&\quad+4 P : \alpha(\alpha^*)^2\alpha^{1*}:  -4 P :\alpha^1  \alpha^{*} (\alpha^{1*})^2:  \\
&\quad -4\delta_{r,0}  P: \alpha^*\alpha^{1*}:\lambda 
  -4\delta_{r,0}  P: \alpha^*\partial\alpha^{1*}:  + 2P\beta:\alpha^*\alpha^{1*}:\\  \\ 
&\quad +   2\chi_0\Big( P:\partial \alpha^*\alpha^{1*}:+P :\alpha^*\partial\alpha^{1*}:+P :\alpha^*\alpha^{1*}:\lambda+\frac{1}{2}(\partial P):\alpha^* \alpha^{1*}:  \Big)\\  \\ 
&\quad    - 2P \beta   \alpha^{*} \alpha^{1*}
-2\kappa_0P\alpha^* \alpha^{1*} \lambda -2\kappa_0P\partial \alpha^* \alpha^{1*}   \\  \\
&\quad  -\beta^1(\alpha^*)^2 
 	-P\beta^1 (\alpha^{1*} )^2  -\kappa_0\left(2P\alpha^*\alpha^{1*}\lambda+2P\alpha^* \partial\alpha^{1*}+\partial  P\alpha^*\alpha^{1*} \right)    \\
&\quad 
+2 \chi_0P \alpha^{*} \alpha^{1*}\lambda     \\ \\ 
&=   -4\delta_{r,0}P:   \alpha^{*} \alpha^{1*}:\lambda-4\delta_{r,0}P:  \partial( \alpha^{*}) \alpha^{1*}:  \\
	&\hskip 100pt  +\chi_1\left(2:\alpha^*\partial \alpha^*:+:(\alpha^*)^2:\lambda  \right) \\  \\ 
&\quad  -4\delta_{r,0}  P:\alpha^{1*} \alpha^*:\lambda-4\delta_{r,0}\partial P:\alpha^{1*} \alpha^*:-4\delta_{r,0} P:\partial\alpha^{1*} \alpha^*:
  \\
&\quad
-4\delta_{r,0}P: \alpha^*\alpha^{1*}:\lambda  -4\delta_{r,0}P: \partial(\alpha^*)\alpha^{1*}: \\
&\quad -4\delta_{r,0}  P: \alpha^*\alpha^{1*}:\lambda 
    -4\delta_{r,0}  P: \alpha^*\partial\alpha^{1*}:  \\  \\ 
&\quad +   2\chi_0\Big( P:\partial \alpha^*\alpha^{1*}:+P :\alpha^*\partial\alpha^{1*}:+P :\alpha^*\alpha^{1*}:\lambda+\frac{1}{2}(\partial P):\alpha^* \alpha^{1*}:  \Big)\\  \\ 
&\quad   
-\kappa_0P\alpha^* \alpha^{1*} \lambda -\kappa_0P\partial \alpha^* \alpha^{1*}   \\  \\
&\quad  -\kappa_0\left(2P\alpha^*\alpha^{1*}\lambda+2P\alpha^* \partial\alpha^{1*}+\partial  P\alpha^*\alpha^{1*} \right)    \\
&\quad
 -2 \chi_0P \alpha^{*} \alpha^{1*}\lambda  \\
 &=0.
\end{align*}

\end{proof}

\begin{acknowledgement}
The first author would like to thank the other two authors and the University of S\~ao Paulo for hosting him while he visited Brazil in June of 2013 where part of this work was completed.   The second author was partially supported by Fapesp (2010/50347-9) and CNPq (301743/2007-0). The third author was supported by FAPESP (2012/02459-8).
\end{acknowledgement}
\def\cprime{$'$} \def\cprime{$'$}

 \end{document}